\documentclass[a4paper,12pt]{amsart}

\usepackage{amssymb}
\usepackage{graphicx}
\usepackage{float}
\usepackage{amsmath}
\usepackage{array}
\usepackage{multirow}
\usepackage{mathrsfs}
\usepackage{textcomp}

\newcolumntype{M}[1]{>{\raggedright}m{#1}}

\DeclareMathAlphabet{\mathpzc}{OT1}{pzc}{m}{it}

\newtheorem{theorem}{Theorem}[section]

\newtheorem{proposition}[theorem]{Proposition}

\newtheorem{observation}[theorem]{Observation}

\theoremstyle{definition}
\newtheorem{definition}[theorem]{Definition}
\newtheorem{example}[theorem]{Example}

\theoremstyle{remark}
\newtheorem{remark}[theorem]{Remark}

\numberwithin{equation}{section}

\allowdisplaybreaks[4]

\begin{document}

\title{On the fundamental groups of non-generic $\mathbb{R}$-join-type curves}

\author{Christophe Eyral and Mutsuo Oka}

\address{C. Eyral, Institute of Mathematics, Polish Academy of Sciences, ul. \'Sniadeckich 8, P.O. Box 21, 00-956 Warsaw, Poland}  
\email{eyralchr@yahoo.com} 
\address{M. Oka, Department of Mathematics, Tokyo University of Science, 1-3 Kagurazaka, Shinjuku-ku, Tokyo 162-8601, Japan}   
\email{oka@rs.kagu.tus.ac.jp}

\thanks{This research was partially supported by the \emph{Institut des Hautes \'Etudes Scientifiques}, Bures-sur-Yvette, France.}


\dedicatory{Dedicated to S. Papadima and A. Dimca for their 60th birthday}

\begin{abstract}
An \emph{$\mathbb{R}$-join-type curve} is a curve in $\mathbb{C}^2$ defined by an equation of the form
\begin{equation*}
a\cdot\prod_{j=1}^\ell (y-\beta_j)^{\nu_j} =
b\cdot\prod_{i=1}^m (x-\alpha_i)^{\lambda_i},
\end{equation*}
where the coefficients $a$, $b$, $\alpha_i$ and $\beta_j$ are \emph{real} numbers. For generic values of $a$ and $b$, the singular locus of the curve consists of the points $(\alpha_i,\beta_j)$ with $\lambda_i,\nu_j\geq 2$ (so-called \emph{inner} singularities). In the non-generic case, the inner singularities are not the only ones: the curve may also have \emph{`outer'} singularities. The fundamental groups of (the complements of) curves having only inner singularities are considered in \cite{O}. In the present paper, we investigate the fundamental groups of a special class of curves possessing outer~singularities.
\end{abstract}

\maketitle

\markboth{C. Eyral and M. Oka}{On the fundamental groups of non-generic $\mathbb{R}$-join-type curves}  

\section{Introduction}\label{intro}

Let $\nu_1,\ldots,\nu_\ell,\lambda_1,\ldots,\lambda_m$ be positive integers. Denote by $\nu_0$ (respectively, $\lambda_0$) the greatest common divisor of $\nu_1,\ldots,\nu_\ell$ (respectively, of $\lambda_1,\ldots,\lambda_m$). Set $d:=\sum_{j=1}^\ell \nu_j$ and $d':=\sum_{i=1}^m \lambda_i$.
A curve $C$ in $\mathbb{C}^2$ is called a \emph{join-type curve with exponents} $(\nu_1,\ldots,\nu_\ell;\lambda_1,\ldots,\lambda_m)$
if it is defined by an equation of the form $f(y)=g(x)$, where
\begin{equation}\label{defining.equation}
f(y):=a\cdot\prod_{j=1}^\ell (y-\beta_j)^{\nu_j} 
\quad\mbox{and}\quad 
g(x):=b\cdot\prod_{i=1}^m (x-\alpha_i)^{\lambda_i}.
\end{equation}
Here, $a$ and $b$ are non-zero complex numbers, and $\beta_1,\ldots,\beta_\ell$ (respectively, $\alpha_1,\ldots,\alpha_m$) are mutually distinct complex numbers.
We say that $C$ is an $\mathbb{R}$-join-type curve if the coefficients $a$, $b$, $\alpha_i$ ($1\leq i\leq m$) and $\beta_j$ ($1\leq j\leq \ell$) are \emph{real} numbers.

The singular points of $C$ (i.e., the points $(x,y)$ satisfying $f(y)=g(x)$ and $f'(y)=g'(x)=0$) divide into two categories: the points $(x,y)$ which also satisfy the equations $f(y)=g(x)=0$, and those for which $f(y)\not=0$ and $g(x)\not=0$.
Clearly, the singular points contained in the intersection of lines $f(y)=g(x)=0$ are the points $(\alpha_i,\beta_j)$~with $\lambda_i,\nu_j\geq 2$. Hereafter, such singular points will be called \emph{inner} singularities, while the singular points $(x,y)$ with $f(y)\not=0$ and $g(x)\not=0$ will be called \emph{outer} or \emph{exceptional} singularities. It is easy to see that the singular points of a join-type curve are Brieskorn--Pham singularities~$\mathbf{B}_{\nu,\lambda}$ (normal form $y^{\nu}-x^{\lambda}$). For example, inner singularities are of type $\mathbf{B}_{\nu_j,\lambda_i}$. In the case of $\mathbb{R}$-join-type curves, we shall see, more specifically, that outer singularities can be only node singularities (i.e., Brieskorn--Pham singularities of type $\mathbf{B}_{2,2}$).

Clearly, for generic values of $a$ and $b$, under any fixed choice of the coefficients $\alpha_i$ ($1\leq i\leq m$) and $\beta_j$ ($1\leq j\leq \ell$), the curve $C$ has only inner singularities. In this case, it is shown in \cite{O} that the fundamental group $\pi_1(\mathbb{C}^2\setminus C)$ is isomorphic to the group $G(\nu_0;\lambda_0)$ obtained by taking $p=\nu_0$ and $q=\lambda_0$ in the presentation (\ref{presentationGpq}) below.
(In \cite{O} it is assumed that $d=d'$ but the same proof works for $\pi_1(\mathbb{C}^2\setminus C)$ and $d\not=d'$.)
For example, if $C$ has only inner singularities and if $\lambda_0$ or $\nu_0$ is equal to 1, then $\pi_1(\mathbb{C}^2 \setminus C)\simeq\mathbb{Z}$. 

In the present paper, we prove that the result of \cite{O} extends to~certain $\mathbb{R}$-join-type curves possessing \emph{outer} singularities.\footnote{Note that if $C$ is a join-type curve with \emph{non-real} coefficients and with only inner singularities, then it can always be deformed to an $\mathbb{R}$-join-type curve $C_1$ by a deformation $\{C_t\}_{0\leq t\leq 1}$ such that $C_0=C$ and $C_t$ is a join-type curve with only inner singularities and with the same exponent as $C$ (cf.~\cite{O}). (In particular, the topological type of $C_t$ (respectively, $\mathbb{C}^2\setminus C_t$) is independent of $t$.) For curves possessing outer singularities, this is no longer true in general.}  These curves are~defined as follows. Let $C$ be an $\mathbb{R}$-join-type curve. Then, without loss of generality, we can assume that the real numbers $\alpha_i$ ($1\leq i\leq m$) and $\beta_j$ ($1\leq j\leq \ell$) are indexed so that $\alpha_1<\ldots <\alpha_m$ and $\beta_1<\ldots <\beta_\ell$. 
Then, by considering the restriction of the function $g(x)$ to the real numbers, we see easily that the equation $g'(x)=0$ has at least one real root $\gamma_i$ in the open interval $(\alpha_i,\alpha_{i+1})$ for each $i=1,\ldots, m-1$. Since the degree~of 
\begin{displaymath}
g'(x)\bigg/ \prod_{i=1}^m (x-\alpha_i)^{\lambda_i-1}
\end{displaymath}
is $m-1$, it follows that the roots of $g'(x)=0$ are exactly $\gamma_1, \ldots, \gamma_{m-1}$ and the coefficients $\alpha_i$ with $\lambda_i\geq 2$ (cf.~Figure \ref{generalgraph}). In particular, this shows that $\gamma_1, \ldots, \gamma_{m-1}$ are \emph{simple} roots of $g'(x)=0.$ Similarly, the equation $f'(y)=0$ has $\ell-1$ simple roots $\delta_1,\ldots,\delta_{\ell-1}$ such that $\beta_j<\delta_j<\beta_{j+1}$ for each $j=1,\ldots,\ell-1$. The other roots of $f'(y)=0$ are the coefficients $\beta_j$ with $\nu_j\geq 2$. (They are simple for $\nu_j=2$.)

\begin{figure}[H]
\includegraphics[width=12cm,height=3.5cm]{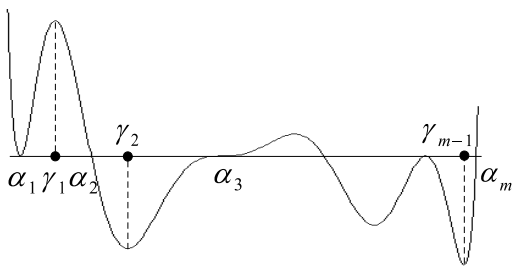}
\caption{Real graph of $g$}
\label{generalgraph}
\end{figure}

We fix the following terminology.

\begin{definition}
\begin{enumerate}
\item
We say that the curve $C$ is \emph{generic} if it has only inner singularities. In other words, $C$ is generic if and only if, for any $1\leq i\leq m-1$, $g(\gamma_i)$ is a regular value for $f$ (i.e., $g(\gamma_i)\not=f(\delta_j)$ for any $1\leq j\leq \ell-1$).  (Of course, this is also equivalent to the condition that, for any $1\leq j\leq \ell-1$, $f(\delta_j)$ is a regular value for $g$.)
\item
We say that $C$ is \emph{semi-generic with respect to $g$} if there exists an integer $i_0$ ($1\leq i_0\leq m$) such that $g(\gamma_{i_0-1})$ and $g(\gamma_{i_0})$ are regular values for $f$. (For $i_0=1$, this condition reduces to $g(\gamma_1)\notin\mathscr{V}_{\mbox{\tiny crit}}(f)$, and for $i_0=m$, it reduces to $g(\gamma_{m-1})\notin\mathscr{V}_{\mbox{\tiny crit}}(f)$, where $\mathscr{V}_{\mbox{\tiny crit}}(f)$ is the set of critical values of $f$.) The semi-genericity \emph{with respect to $f$} is defined similarly by exchanging the roles of $f$ and~$g$.
\end{enumerate}
\end{definition}

\begin{figure}[H]
\includegraphics[width=12cm,height=3.5cm]{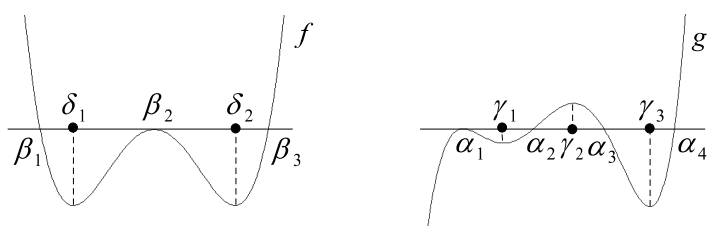}
\caption{Semi-genericity with respect to $g$ ($i_0=1$ or $2$)}
\label{generalgraph2}
\end{figure}

\begin{remark}
It is obvious that a generic curve is also semi-generic with respect to both $g$ and $f$, while the converse is not true.
Also, note that $C$ can be semi-generic with respect to $g$ without being semi-generic with respect to $f$. For example, consider the curve defined by the polynomials $f$ and $g$ given in Figure \ref{generalgraph2}, where $f(\delta_1)=f(\delta_2)=g(\gamma_3)$. 
\end{remark}

Here is our main result.

\begin{theorem}\label{mt}
Let $C$ be an $\mathbb{R}$-join-type curve in $\mathbb{C}^2$ defined by the equation $f(y)=g(x)$, where $f$ and $g$ are as in (\ref{defining.equation}). If $C$ is semi-generic with respect to $g$, then 
\begin{displaymath}
\pi_1(\mathbb{C}^2\setminus C)\simeq G(\nu_0;\lambda_0),
\end{displaymath}
where, as above, $\nu_0:=\gcd(\nu_1,\ldots,\nu_\ell)$, $\lambda_0:=\gcd(\lambda_1,\ldots,\lambda_m)$ 
and $G(\nu_0;\lambda_0)$ is the group obtained by taking $p=\nu_0$ and $q=\lambda_0$ in the presentation (\ref{presentationGpq}). \par
Furthermore, if $\widetilde C$ is the projective closure of $C$, then 
\begin{equation*}
\pi_1(\mathbb{P}^2\setminus \widetilde C)\simeq
\left\{
\begin{aligned}
& G(\nu_0;\lambda_0;d/\nu_0) & \mbox{if }\  d\geq d',\\
& G(\lambda_0;\nu_0;d'/\lambda_0) & \mbox{if }\  d'\geq d,
\end{aligned}
\right.
\end{equation*}
where $G(\nu_0;\lambda_0;d/\nu_0)$ (respectively, $G(\lambda_0;\nu_0;d'/\lambda_0)$) is the group obtained by taking $p=\nu_0$, $q=\lambda_0$ and $r=d/\nu_0$ (respectively, $p=\lambda_0$, $q=\nu_0$ and $r=d'/\lambda_0$) in the presentation (\ref{presentationGpqr}).
\end{theorem}

\begin{remark}\label{ramt}
The conclusions of Theorem \ref{mt} are still valid if we suppose that $C$ is semi-generic with respect to $f$. This is an immediate consequence of the theorem itself and Proposition \ref{isoGpqGqp} below.
\end{remark}

\begin{example}
With the same hypotheses as in Theorem \ref{mt}, if $d$ is a prime number and $\ell\geq 2$, then $\nu_0=1$, and hence $\pi_1(\mathbb{C}^2\setminus C)$ is isomorphic to $G(1;\lambda_0)\simeq\mathbb{Z}$ while $\pi_1(\mathbb{P}^2\setminus \widetilde C)$ is isomorphic to $\mathbb{Z}_d$ or $\mathbb{Z}_{d'}$ depending on whether $d\geq d'$ or $d'\geq d$. (Of course, if $d'$ is a prime number and $m\geq 2$, then $\lambda_0=1$, and we get the same conclusion.) 
\end{example}

\section{The groups $G(p;q)$ and $G(p;q;r)$}\label{groups}

Let $p,q,r$ be positive integers. In this section, we recall the definitions and collect the basic properties of the groups $G(p;q)$ and $G(p;q;r)$ introduced in \cite{O} and which appear in Theorem \ref{mt} as the fundamental groups of the affine and the projective semi-generic $\mathbb{R}$-join-type curves, respectively. 
 
The group $G(p;q)$ is defined by the presentation
\begin{equation}\label{presentationGpq}
\langle\, \omega,\, a_k \ (k\in\mathbb{Z}) \mid 
\omega=a_{p-1}a_{p-2}\ldots a_0,\ \mathscr{R}_{q,k},\ \mathscr{R}'_{p,k}\ (k\in\mathbb{Z})\, \rangle,
\end{equation}
where 
\begin{align*}
& \mathscr{R}_{q,k}\colon a_{k+q}=a_k \mbox{ (periodicity relation)};\\
& \mathscr{R}'_{p,k}\colon a_{k+p}=\omega a_{k} \omega^{-1} \mbox{ (conjugacy relation)}.
\end{align*} 

The following proposition is used to show that the conclusions of Theorem \ref{mt} still hold if we suppose that $C$ is semi-generic with respect to $f$ (cf.~Remark \ref{ramt}).

\begin{proposition}\label{isoGpqGqp}
The groups $G(p;q)$ and $G(q;p)$ are isomorphic.
\end{proposition}

\begin{proof}
From a purely algebraic point of view, this proposition is not obvious. However, by \cite{O}, we know that if $C$ is the generic join-type curve $y^p=x^q$, then $\pi_1(\mathbb{C}^2\setminus C)\simeq G(p;q)$. Now, by exchanging the roles of $y$ and $x$, we also have $\pi_1(\mathbb{C}^2\setminus C)\simeq G(q;p)$. 
\end{proof}

The following proposition will be useful to prove Theorem \ref{mt}.

\begin{proposition}[cf.~\cite{O}]\label{prop26}
The relations $\mathscr{R}'_{p,k}$ ($k\in\mathbb{Z}$) and $\omega=a_{p-1}a_{p-2}\ldots a_0$ imply the following new relation for any $k\in\mathbb{Z}$:
\begin{align*}
\omega=a_k a_{k-1}\ldots a_{k-p+1}.
\end{align*} 
\end{proposition}

It follows from this proposition that, for any $n\in\mathbb{Z}$, we can reorder the generators as $b_k:=a_{k+n}$ without changing the relations. That is, we have $\omega=b_{p-1}b_{p-2}\ldots b_0$, $b_{k+q}=b_k$ and $b_{k+p}=\omega b_{k} \omega^{-1}$ ($k\in\mathbb{Z}$). 

Now, let $q_1,\ldots,q_n$ be positive integers, and $G(p;\{q_1,\ldots,q_n\})$ the group defined by the presentation
\begin{align*}
\langle\, \omega,\, a_k \ (k\in\mathbb{Z}) \mid 
\omega=a_{p-1}a_{p-2}\ldots a_0,\ 
\mathscr{R}_{q_i,k},\ \mathscr{R}'_{p,k}\ (1\leq i\leq n,\, k\in\mathbb{Z})\, \rangle,
\end{align*}
where 
\begin{align*}
& \mathscr{R}_{q_i,k}\colon a_{k+q_i}=a_k.
\end{align*}

We shall also use the next result in the proof of Theorem \ref{mt}.

\begin{proposition}[cf.~\cite{O}]\label{prop29}
The group $G(p;\{q_1,\ldots,q_n\})$ is isomorphic to the group $G(p;q_0)$, where $q_0:=\gcd(q_1,\ldots,q_n)$.
\end{proposition}

The next proposition gives necessary and sufficient conditions for the group $G(p;q)$ to be abelian. Thus, it can be used to test the commutativity of the group $\pi_1(\mathbb{C}^2\setminus C)$ which appears in Theorem \ref{mt}.

\begin{proposition}[cf.~\cite{O}]\label{cor312}
The group $G(p;q)$ is abelian if and only if $q=1$ or $p=1$ or $p=q=2$. More precisely,
\begin{displaymath}
G(p;q)\simeq \left\{
\begin{aligned}
& \mathbb{Z} &&\mbox{if} && q=1 \mbox{ or } p=1; \\
& \mathbb{Z}\times\mathbb{Z} &&\mbox{if} && p=q=2.
\end{aligned}
\right.
\end{displaymath}
\end{proposition}

The group $G(p;q;r)$ is defined by the presentation
\begin{align}\label{presentationGpqr}
\langle\, \omega,\, a_k \ (k\in\mathbb{Z}) \mid 
\omega=a_{p-1}a_{p-2}\ldots a_0,\
\omega^r=e,\ \mathscr{R}_{q,k},\ \mathscr{R}'_{p,k}\ (k\in\mathbb{Z})\, \rangle,
\end{align}
where $e$ is the unit element. In other words, $G(p;q;r)$ is the quotient of $G(p;q)$ by the normal subgroup generated by $\omega^{r}$. 

The next proposition is an interesting special case.

\begin{proposition}[cf.~\cite{O}]
If $\gcd(p,q)=1$ and $r=q$, then $G(p;q;q)$ is isomorphic to the free product $\mathbb{Z}_p*\mathbb{Z}_q$.
\end{proposition}

Finally, we conclude this section with the following proposition~which gives necessary and sufficient conditions for the group $G(p;q;r)$ to be abelian. This proposition can be used to test the commutativity of the group $\pi_1(\mathbb{P}^2\setminus \widetilde C)$ which appears in Theorem \ref{mt}.

\begin{proposition}[cf.~\cite{O}]\label{cor216}
The group $G(p;q;r)$ is abelian if and only if one of the following conditions is satisfied:
\begin{enumerate}
\item
$\gcd(p,q)=\gcd(q,r)=1$;
\item
$p=1$;
\item
$\gcd(p,q)=2$, $\gcd(q/2,r)=1$ and $p=2$.
\end{enumerate}
More precisely,
\begin{displaymath}
G(p;q;r)\simeq \left\{
\begin{aligned}
& \mathbb{Z}_{pr} &&\mbox{if} && \gcd(p,q)=\gcd(q,r)=1; \\
& \mathbb{Z}_{r} &&\mbox{if} && p=1; \\
& \mathbb{Z}\times\mathbb{Z}_{r} &&\mbox{if} && \gcd(p,q)=2,\ \gcd(q/2,r)=1 \mbox{ and } p=2. 
\end{aligned}
\right.
\end{displaymath}
\end{proposition}

\section{Special pencil lines}\label{sectionspl}

To compute the fundamental group $\pi_1(\mathbb{C}^2\setminus C)$ in Theorem \ref{mt}, we use the Zariski--van Kampen theorem with the pencil $\mathscr{P}$ given by the vertical lines $L_{\gamma}\colon x=\gamma$, where $\gamma\in\mathbb{C}$ (cf.~\cite{vK,O2,Z}).\footnote{Note that this pencil is `admissible' in the sense of \cite{O2}.} 
This theorem says that
\begin{displaymath}
\pi_1(\mathbb{C}^2\setminus C)\simeq\pi_1(L_{\gamma_0}\setminus C)\big/\mathscr{M},
\end{displaymath} 
where $L_{\gamma_0}$ is a generic line of $\mathscr{P}$ and $\mathscr{M}$ is the normal subgroup of $\pi_1(L_{\gamma_0}\setminus C)$ generated by the monodromy relations associated with the `special' lines of $\mathscr{P}$. Here, a line $L_\gamma$ of $\mathscr{P}$ is called \emph{special} if it meets the curve $C$ at a point $(\gamma,\delta)$ with intersection multiplicity at least $2$. This happens if and only if $f(\delta)=g(\gamma)$ and $f'(\delta)=0$. 

Let $\gamma_{j,1},\ldots,\gamma_{j,d'}$ be the roots of $g(x)=f(\delta_j)$ for $1\leq j\leq \ell-1$, where $\delta_j$ is defined as in Section \ref{intro}.
If $g'(\gamma_{j,k})\not=0$, then $(\gamma_{j,k},\delta_j)$ is a simple point of $C$. In a small neighbourhood of this point, $C$ is topologically described by 
\begin{align}\label{nearsimple}
(y-\delta_{j})^2 = c(x-\gamma_{j,k}),
\end{align}
where $c\not=0$, and the line $x=\gamma_{j,k}$ is tangent to the curve at $(\gamma_{j,k},\delta_{j})$ with intersection multiplicity $2$. (We recall that $\delta_j$ is a \emph{simple} root of $f'(y)=0$.) This is the case if $\gamma_{j,k}\in\mathbb{C}\setminus \mathbb{R}$, as $g'(x)=0$ has only real roots. 
If $g'(\gamma_{j,k})=0$, then $(\gamma_{j,k},\delta_j)$ is an outer singularity of type $\textbf{B}_{2,2}$. (Note that $\gamma_{j,k}$ is a simple root of $g'(x)=0$.) Near this point, the curve is topologically equivalent to
\begin{align}\label{nearA1}
(y-\delta_j)^2 = c(x-\gamma_{j,k})^2.
\end{align}
For each $\beta_j$ with $\nu_j\geq 2$, the roots of $g(x)=f(\beta_j)$ are $\alpha_1,\ldots,\alpha_m$. If $\lambda_i=1$, then $(\alpha_i,\beta_j)$ is a simple point of $C$. In a small neighbourhood of it, $C$ is topologically given by
\begin{equation}\label{nearsimpleb}
(y-\beta_j)^{\nu_j}=c(x-\alpha_i),
\end{equation}
and the line $x=\alpha_i$ is tangent to $C$ at $(\alpha_i,\beta_j)$ with intersection multiplicity $\nu_j$.
If $\lambda_i\geq 2$, then the point $(\alpha_i,\beta_j)$ is an inner singularity of type $\textbf{B}_{\nu_j,\lambda_i}$, and in a small neighbourhood of it, the curve is topologically equivalent to
\begin{equation}\label{ngs}
(y-\beta_j)^{\nu_j}=c(x-\alpha_i)^{\lambda_i}.
\end{equation}

\section{Bifurcation graph}\label{spl}

The special lines of the pencil $\mathscr{P}$ correspond to certain vertices of a graph called the `bifurcation graph'. This graph is defined as follows.
Let $\mathscr{V}_{\mbox{\tiny crit}}(f)$ (respectively, $\mathscr{V}_{\mbox{\tiny crit}}(g)$) be the set of critical values of $f$ (respectively, $g$), and let $\mathscr{V}_{\mbox{\tiny crit}}:=\mathscr{V}_{\mbox{\tiny crit}}(f)\cup \mathscr{V}_{\mbox{\tiny crit}}(g)$. We have $\mathscr{V}_{\mbox{\tiny crit}} =\{0,g(\gamma_1),\ldots,g(\gamma_{m-1}),f(\delta_1),\ldots,f(\delta_{\ell-1})\}$.
Denote by $\Sigma$ the bamboo-shaped graph (embedded in the real axis) whose vertices are the points of $\mathscr{V}_{\mbox{\tiny crit}}$ (cf.~Figure \ref{BambooGeneral}). This graph can be decomposed into two connected subgraphs $\Sigma_+$ and $\Sigma_-$, where $\Sigma_+$ (respectively, $\Sigma_-$) is the subgraph whose vertices are $\geq 0$ (respectively,~$\leq 0$). Hereafter, we shall denote by $v_+:=\sup\, \{v\mid v\in\mathscr{V}_{\mbox{\tiny crit}}\}$ and $v_-:=\inf\, \{v\mid v\in\mathscr{V}_{\mbox{\tiny crit}}\}$.
The pull-back graph $\Gamma:=g^{-1}(\Sigma)$ of $\Sigma$ by $g$ is called the \emph{bifurcation graph (or `dessin d'enfants') associated with the curve $C$ with respect to $g$}. Its vertices are the points of the set $g^{-1}(\mathscr{V}_{\mbox{\tiny crit}})$.

\begin{observation}
The special lines $x=\gamma$ of the pencil $\mathscr{P}$ are given by the vertices $\gamma$ of $\Gamma$ such that $g(\gamma)\in\mathscr{V}_{\mbox{\tiny \emph{crit}}}(f)$. 
\end{observation}

\begin{figure}[H]
\includegraphics[width=10.9cm,height=1.4cm]{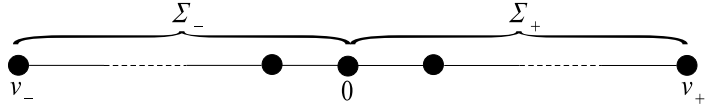}
\caption{Graph $\Sigma$}
\label{BambooGeneral}
\end{figure}

The bifurcation graph uniquely decomposes as the union of connected subgraphs $\Gamma(\alpha_1),\ldots,\Gamma(\alpha_m)$ such that, for $1\leq i\leq m$, the~following properties are satisfied:
\begin{enumerate}
\item
$\Gamma(\alpha_i)$ is a star-shaped graph with `centre' $\alpha_i$, and with $2\lambda_i$ branches (respectively, $\lambda_i$ branches) if $v_+>0$ and $v_-<0$ (respectively, if $v_+$ or $v_-$ is zero);
\item 
the restriction of $g$ to $\Gamma(\alpha_i)$ is an $\lambda_i$-fold branched covering onto $\Sigma$, whose branched locus is $\{0\}$, and $g^{-1}(0)\cap\Gamma(\alpha_i)=\{\alpha_i\}$;
\item
for $i\not=m$, $\Gamma(\alpha_i)\cap \Gamma(\alpha_{i+1})=\{\gamma_i\}$, and if $g(\gamma_i)\notin\{v_{-},v_{+}\}$, then the branch of $\Gamma(\alpha_i)$ (respectively, $\Gamma(\alpha_{i+1})$) with $\gamma_i$ as a vertex goes vertically downward (respectively, vertically upward) at $\gamma_i$.
\end{enumerate}

\begin{definition}
The subgraphs $\Gamma(\alpha_i)$ ($1\leq i\leq m$) are called the \emph{satellite graphs} of $\Gamma$. We say that a satellite $\Gamma(\alpha_i)$ is \emph{regular} if $g(\gamma_{i-1})$ and $g(\gamma_{i})$ are regular values for $f$. (For $i=1$, this condition reduces to $g(\gamma_{1}) \notin \mathscr{V}_{\mbox{\tiny crit}}(f)$, and for $i=m$, it reduces to $g(\gamma_{m-1}) \notin \mathscr{V}_{\mbox{\tiny crit}}(f)$.) 
\end{definition}

Clearly, the curve $C$ is generic if and only if all the satellites subgraphs of $\Gamma$ are regular. The curve is semi-generic with respect to $g$ if and only if $\Gamma$ has at least one regular satellite.

\begin{figure}[H]
\includegraphics[width=11.6cm,height=3.4cm]{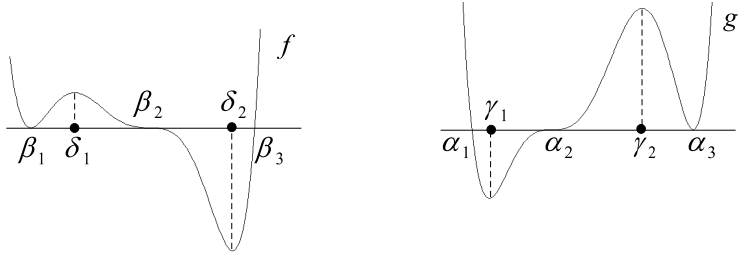}
\caption{Real graphs of $f$ and $g$ (Example \ref{examplebifurcationgraph})}
\label{Example1-graphs}
\end{figure}

\begin{figure}[H]
\includegraphics[width=10.9cm,height=1.4cm]{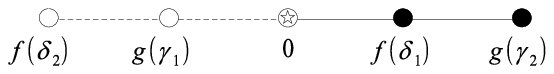}
\caption{Graph $\Sigma$ (Example \ref{examplebifurcationgraph})}
\label{Bamboo}
\end{figure}

\begin{figure}[H]
\includegraphics[width=10.9cm,height=3.4cm]{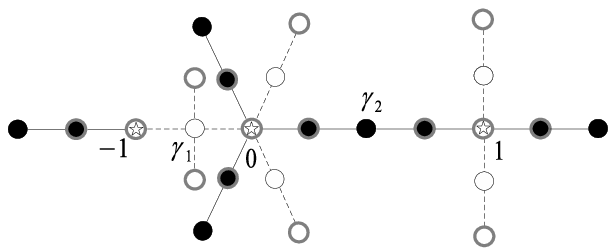}
\caption{Bifurcation graph $\Gamma$ (Example \ref{examplebifurcationgraph})}
\label{exampleBG}
\end{figure}

\begin{figure}[H]
\includegraphics[width=10.9cm,height=3.4cm]{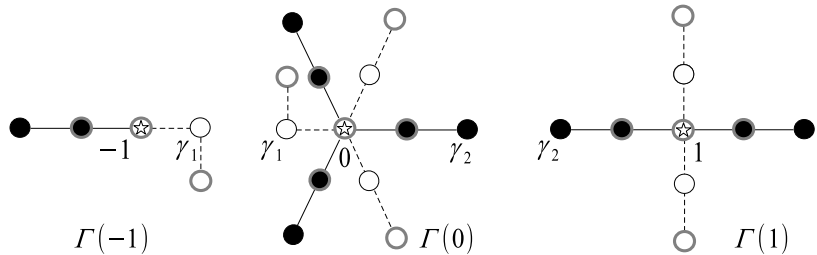}
\caption{Satellites $\Gamma(-1)$, $\Gamma(0)$ and $\Gamma(1)$ (Example \ref{examplebifurcationgraph})}
\label{examplesatellites}
\end{figure}

\begin{example}\label{examplebifurcationgraph}
Consider the $\mathbb{R}$-join-type curve $C$ defined by the polynomials $f(y)=(y+1)^2 y^3 (y-2)$ and $g(x)=2(x+1)x^3(x-1)^2$. Then, $f$ has four critical points $\beta_1=-1$,  $\delta_1=(1-\sqrt{5})/2$, $\beta_2=0$ and $\delta_2=(1+\sqrt{5})/2$. The polynomial $g$ also has four critical points $\gamma_1=-(1+\sqrt{73})/12$, $\alpha_2=0$, $\gamma_2=-(1-\sqrt{73})/12$ and $\alpha_3=1$. 
See Figure \ref{Example1-graphs}. (In the figure, the numerical scale is not respected; however, the order $f(\delta_2)<g(\gamma_1)<f(\delta_1)<g(\gamma_2)$ is rigorously respected.)
As $g(\gamma_i)\not=f(\delta_j)$ for any $1\leq i,j\leq 2$, the curve $C$ is generic. The corresponding graphs $\Sigma$ and $\Gamma$ are given in Figures \ref{Bamboo} and \ref{exampleBG} respectively. The satellites $\Gamma(-1)$, $\Gamma(0)$ and $\Gamma(1)$ associated with $\Gamma$ are given in Figure \ref{examplesatellites}.  The black vertices and the full lines of the bifurcation graph $\Gamma$ correspond to the part above the positive branch $\Sigma_+$ of $\Sigma$. The white vertices and the dotted lines correspond to the part above the negative branch $\Sigma_-$. The star-style vertices represent the points $\alpha_1=-1$, $\alpha_2=0$ and $\alpha_3=1$ which are the centres of the satellites. All the vertices (black, white and star-style) surrounded with a gray circle correspond to the special lines of the pencil $\mathscr{P}$. As the curve is generic, we have $\pi_1(\mathbb{C}^2\setminus C)\simeq G(1;1)\simeq \mathbb{Z}$ and $\pi_1(\mathbb{P}^2\setminus \widetilde C)\simeq G(1;1;6)\simeq \mathbb{Z}_6$ (by \cite{O} or Theorem \ref{mt} above).
\end{example}

Now, let us give an example with a semi-generic curve which is not generic.

\begin{example}\label{examplebifurcationgraph-ex2}
Consider the $\mathbb{R}$-join-type curve $C$ defined by the polynomials $f(y)=c (y+1) y^3 (y-1)$ and $g(x)=(x+1)^2 x^3(x-2)$, where the coefficient $c$ is positive. We check easily that $f$ has three critical points  $\delta_1<\beta_2=0<\delta_2$, while the polynomial $g$ has four critical points $\alpha_1=-1<\gamma_1<\alpha_2=0<\gamma_2$. We choose the coefficient $c$ so that $f(\delta_2)=g(\gamma_2)$ (in particular, $C$ is not generic). 
See Figure~\ref{Example2-graphs}. (Again, the figure is not numerically correct but the order $f(\delta_2)=g(\gamma_2)<g(\gamma_1)<f(\delta_1)$ is respected.) Now, as $g(\gamma_1)$ is a regular value for $f$, the satellite $\Gamma(\alpha_1)$ is regular, and therefore the curve $C$ is semi-generic with respect to $g$. The corresponding graphs $\Sigma$ and $\Gamma$ are given in Figures \ref{Bamboo-ex2} and \ref{exampleBG-ex2} respectively. The satellites $\Gamma(-1)$, $\Gamma(0)$ and $\Gamma(2)$ associated with $\Gamma$ are given in Figure \ref{examplesatellites-ex2}. (The significance of the colours is as above.) As the curve is semi-generic with respect to $g$, Theorem \ref{mt} says that $\pi_1(\mathbb{C}^2\setminus C)\simeq G(1;1)\simeq \mathbb{Z}$ 
and $\pi_1(\mathbb{P}^2\setminus \widetilde C)\simeq G(1;1;6)\simeq \mathbb{Z}_6$.
\end{example}

\begin{figure}[H]
\includegraphics[width=11.6cm,height=3.4cm]{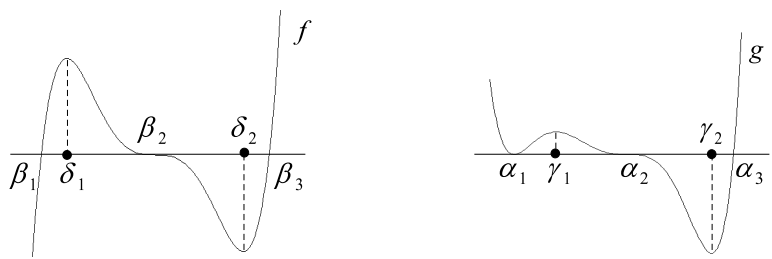}
\caption{Real graphs of $f$ and $g$ (Example \ref{examplebifurcationgraph-ex2})}
\label{Example2-graphs}
\end{figure}

\begin{figure}[H]
\includegraphics[width=10.9cm,height=1.4cm]{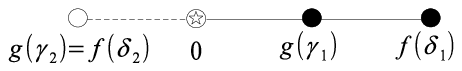}
\caption{Graph $\Sigma$ (Example \ref{examplebifurcationgraph-ex2})}
\label{Bamboo-ex2}
\end{figure}

\begin{figure}[H]
\includegraphics[width=10.9cm,height=3.4cm]{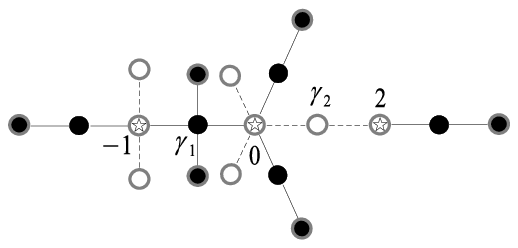}
\caption{Bifurcation graph $\Gamma$ (Example \ref{examplebifurcationgraph-ex2})}
\label{exampleBG-ex2}
\end{figure}

\begin{figure}[H]
\includegraphics[width=10.9cm,height=3.4cm]{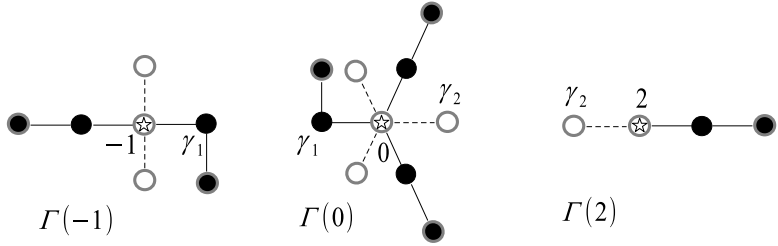}
\caption{Satellites $\Gamma(-1)$, $\Gamma(0)$ and $\Gamma(2)$ (Example \ref{examplebifurcationgraph-ex2})}
\label{examplesatellites-ex2}
\end{figure}

\section{Proof of Theorem \ref{mt}}\label{proofmt}

We suppose that $C$ has at least one exceptional singularity. (When $C$ is generic, the result is already proved in \cite{O}.)
For simplicity, we shall assume $v_-<0$ and $v_+>0$, so that each satellite $\Gamma(\alpha_i)$ ($1\leq i\leq m$) has $2\lambda_i$ branches. (The proof can be easily adapted if $v_-=0$ or $v_+=0$.) As the curve is semi-generic with respect to $g$, there exists $i_0$ 
such that $\Gamma(\alpha_{i_0})$ is a regular satellite (i.e.,  $g(\gamma_{i_0-1})\notin\mathscr{V}_{\mbox{\tiny crit}}(f)$ and $g(\gamma_{i_0})\notin\mathscr{V}_{\mbox{\tiny crit}}(f)$). 

As mentioned above, we use the Zariski--van Kampen theorem with the pencil $\mathscr{P}$ given by the vertical lines $L_{\gamma}\colon x=\gamma$, where $\gamma\in\mathbb{C}$.
We take a sufficiently small positive number $\varepsilon$, and for any real number $\eta$, we write $\eta^-:=\eta-\varepsilon$ and $\eta^+:=\eta+\varepsilon$.
We consider the generic line $L_{\alpha_{i_0}^+}$, and we choose generators $$\xi_{1,0},\ldots,\xi_{1,\nu_1-1},\ldots,\xi_{\ell,0},\ldots,\xi_{\ell,\nu_\ell-1}$$
of the fundamental group $\pi_1(L_{\alpha_{i_0}^+} \setminus C)$ as in Figure~\ref{initalgenerators}. (In the figure, we do not respect the numerical scale; we even zoom on the part that collapses to $\beta_j$ when $\varepsilon \rightarrow 0$.) Here, the loops $\xi_{j,r_j}$ ($1\leq j\leq\ell$, $0\leq r_j\leq\nu_j-1$) are counterclockwise-oriented \emph{lassos} around the intersection points of  $L_{\alpha_{i_0}^+}$ with $C$. We shall refer to these generators as \emph{geometric} generators.

\begin{figure}[H]
\includegraphics[width=11cm,height=5.3cm]{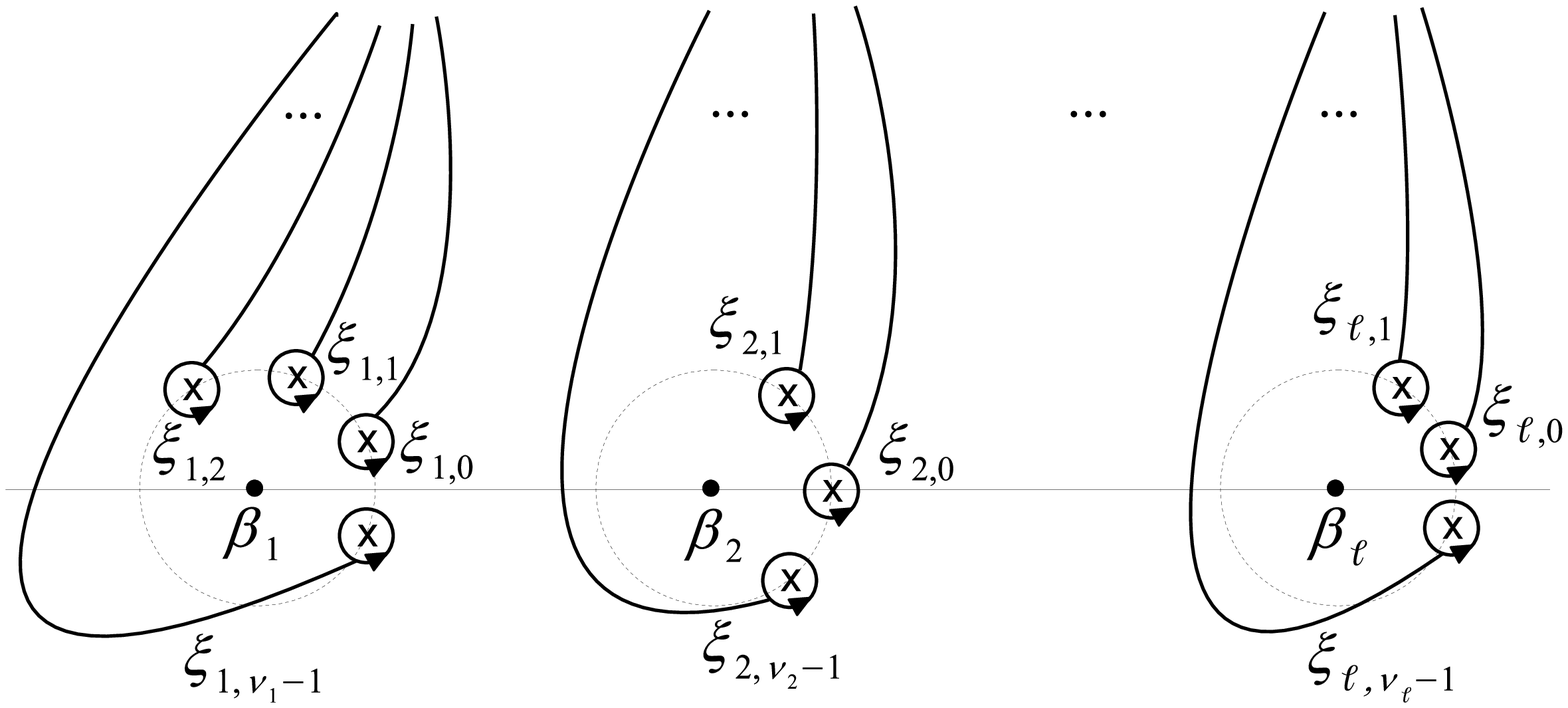}
\caption{Generators of $\pi_1(L_{\alpha_{i_0}^+}\setminus C)$}
\label{initalgenerators}
\end{figure}

For $1\leq j\leq\ell$, $0\leq r_j\leq\nu_j-1$ and $n\in\mathbb{Z}$, let
\begin{align*}
\omega_j:=\xi_{j,\nu_j-1}\ldots\xi_{j,0}
\quad\mbox{and}\quad
\xi_{j,n\nu_j+r_j}:=\omega_j^n \cdot \xi_{j,r_j} \cdot \omega_j^{-n}.
\end{align*}
These relations define elements $\xi_{j,k}$ for any $1\leq j\leq\ell$ and any $k\in\mathbb{Z}$. (Indeed, any $k\in\mathbb{Z}$ can be written as $k=n\nu_j+r_j$, with $n\in\mathbb{Z}$ and $0\leq r_j\leq\nu_j-1$.) It is easy to see that
\begin{equation}\label{rel.1}
\xi_{j,n\nu_j+r}=\omega_j^n \cdot\xi_{j,r} \cdot\omega_j^{-n}
\quad \mbox{for} \quad 1\leq j\leq\ell\mbox{ and }n,r\in\mathbb{Z}.
\end{equation}

\begin{figure}[H]
\includegraphics[width=11cm,height=3.8cm]{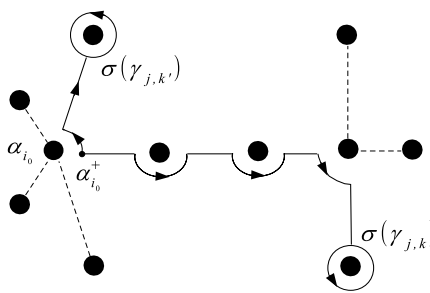}
\caption{Example of standard generators of $\pi_1(\mathbb{C}\setminus \mathscr{S})$}
\label{monodromypaths}
\end{figure}

As usual, to find the monodromy relations associated with the special lines of the pencil $\mathscr{P}$, we~consider a `standard' system of counter\-clockwise-oriented generators of the fundamental group $\pi_1(\mathbb{C}\setminus \mathscr{S})$, where $\mathscr{S}$ is the set consisting of the vertices $\alpha_i$ ($1\leq i\leq m$) and $\gamma_{j,k}$ ($1\leq j\leq \ell-1$, $1\leq k\leq d'$) in the bifurcation graph $\Gamma$.
(We recall that the elements $\gamma_{j,k}$ are the roots of the~equation $g(x)=f(\delta_j)$, where $\delta_j$ is defined as in Section \ref{intro}.)
We choose $\alpha_{i_0}^+$ as base point, and we denote these generators by $\sigma(\alpha_i)$ and $\sigma(\gamma_{j,k})$.
Then, $\sigma(\alpha_i)$ (respectively, $\sigma(\gamma_{j,k})$) is a loop in $\mathbb{C}\setminus \mathscr{S}$ surrounding the vertex $\alpha_i$ (respectively, $\gamma_{j,k}$). It is based at $\alpha_{i_0}^+$ and it runs along the edges of $\Gamma$ avoiding the vertices corresponding to special lines (cf.~Figure \ref{monodromypaths}). The monodromy relations around the special line $L_{\alpha_i}$ (respectively, $L_{\gamma_{j,k}}$) are obtained by moving the generic fibre $L_{\alpha_{i_0}^+} \setminus C$ isotopically `above' the loop $\sigma(\alpha_i)$ (respectively, $\sigma(\gamma_{j,k})$) and by identifying each generator $\xi_{j,r_j}$ ($1\leq j\leq \ell$, $0\leq r_j\leq \nu_j-1$) of the group $\pi_1(L_{\alpha_{i_0}^+} \setminus C)$ with its image by the terminal homeomorphism of this isotopy (cf.~\cite{vK,O2,Z}).

We start with the monodromy relations associated with the special line $L_{\alpha_{i_0}}$. These relations can be found using the local models $y^{\nu_j}=x^{\lambda_{i_0}}$ ($1\leq j\leq\ell$). Precisely, if we write $\lambda_{i_0}=n_j\nu_j+r_j$, $n_j\in\mathbb{Z}$, $0\leq r_j\leq \nu_j-1$, they are given by
\begin{equation*}
\left\{
\begin{aligned}
& \xi_{j,0}=\omega_j^{n_j}\cdot\xi_{j,r_j}\cdot\omega_j^{-n_j},\\
& \xi_{j,1}=\omega_j^{n_j}\cdot\xi_{j,r_j+1}\cdot\omega_j^{-n_j},\\
& \ldots\\
& \xi_{j,\nu_j-(r_j+1)}=\omega_j^{n_j}\cdot\xi_{j,\nu_j-1}\cdot\omega_j^{-n_j},\\
& \xi_{j,\nu_j-r_j}=\omega_j^{n_j+1}\cdot\xi_{j,0}\cdot\omega_j^{-(n_j+1)},\\
& \ldots\\
& \xi_{j,\nu_j-1}=\omega_j^{n_j+1}\cdot\xi_{j,r_j-1}\cdot\omega_j^{-(n_j+1)}.\\
\end{aligned}
\right.
\end{equation*}
By (\ref{rel.1}), these relations can be written more concisely as
\begin{align*}
\xi_{j,k_j}=\omega_j^{n_j}\cdot\xi_{j,k_j+r_j}\cdot\omega_j^{-n_j}=\xi_{j,k_j+\lambda_{i_0}}
\quad\mbox{for}\quad 1\leq j\leq \ell \mbox{ and } 0\leq k_j\leq\nu_j-1.
\end{align*}
In fact, (\ref{rel.1}) shows that 
\begin{align}\label{rel.alpha.i0}
\xi_{j,k}=\xi_{j,k+\lambda_{i_0}} 
\quad \mbox{for} \quad 
1\leq j\leq\ell \mbox{ and } k\in\mathbb{Z}.
\end{align}

\begin{remark}
If $\nu_j=1$ for all $1\leq j\leq\ell$, then $L_{\alpha_{i_0}}$ is not a special line, and hence, the corresponding monodromy relations are trivial. However, it is clear that the relations (\ref{rel.alpha.i0}) remain valid. (Indeed, in this case, $\xi_{j,k}=\xi_{j,0}$ for all $k\in\mathbb{Z}$.)
\end{remark}

Next, we look for the monodromy relations along the branches of $\Gamma(\alpha_{i_0})$. For $0\leq q\leq 2\lambda_{i_0}-1$, we denote by $B_{i_0,q}$ the $q$-th branch of~$\Gamma(\alpha_{i_0})$. We suppose that the branches $B_{i_0,2q}$ (respectively, $B_{i_0,2q+1}$), $0\leq q\leq \lambda_{i_0}-1$, correspond to the positive part $\Sigma_+$ (respectively, the negative part $\Sigma_-$) of $\Sigma$ through the correspondence $\Gamma(\alpha_{i_0})\to\Sigma$ given by the restriction of~$g$. We also suppose that the branch $B_{i_0,0}$ (respectively, $B_{i_0,1}$) contains the line segment $[\alpha_{i_0},\gamma_{i_0}]$  
if $g(\gamma_{i_0})>0$ (respectively, if $g(\gamma_{i_0})<0$). For instance, in the special case of the satellite $\Gamma(\alpha_2)=\Gamma(0)$ of Example \ref{examplebifurcationgraph}, the branches $B_{2,q}$ are as in Figure \ref{labelbranches}.
For simplicity, hereafter we shall suppose $g(\gamma_{i_0})>0$. (The argument is similar in the case $g(\gamma_{i_0})<0$.) 

\begin{figure}[H]
\includegraphics[width=11cm,height=4.2cm]{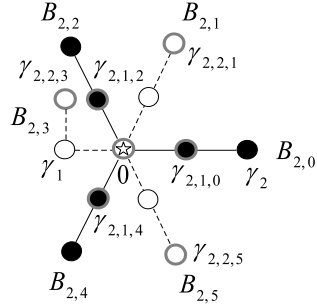}
\caption{Branches of the satellite $\Gamma(\alpha_2)$ of Example \ref{examplebifurcationgraph}}
\label{labelbranches}
\end{figure}

\begin{figure}[H]
\includegraphics[width=13.2cm,height=5cm]{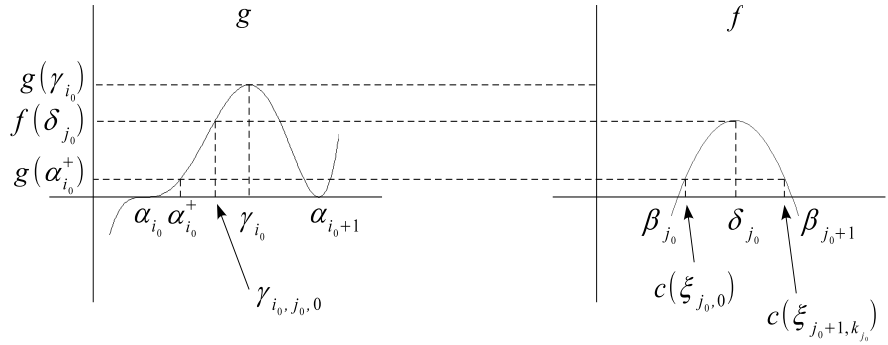}
\caption{Real graphs of $f$ and $g$ when $0<f(\delta_{j_0})<g(\gamma_{i_0})$}
\label{graphs}
\end{figure}

\begin{figure}[H]
\includegraphics[width=11cm,height=5.3cm]{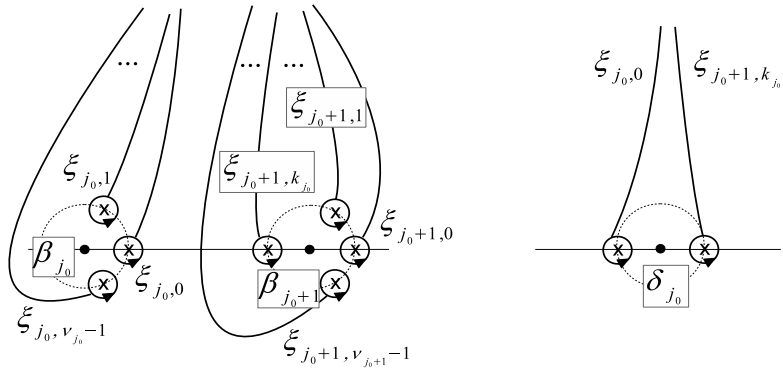}
\caption{Generators at $x=\alpha_{i_0}^+$ (left-side) and at $x=\gamma_{i_0,j_0,0}^-$ (right-side) when $g(\gamma_{i_0})>0$ and $f(\delta_{j_0})>0$}
\label{tangentrelation-case2}
\end{figure}

Pick an element $j_0$ such that $1\leq j_0\leq \ell-1$. If $f(\delta_{j_0})>0$, then, for each $0\leq q\leq \lambda_{i_0}-1$, there exists an unique vertex $\gamma_{i_0,j_0,2q}\in B_{i_0,2q}$ such that $g(\gamma_{i_0,j_0,2q})=f(\delta_{j_0})$. For instance, in the special case of the satellite $\Gamma(\alpha_2)=\Gamma(0)$ of Example \ref{examplebifurcationgraph}, $f(\delta_1)>0$ and for $0\leq q\leq 2$ there exists an unique vertex $\gamma_{2,1,2q}\in B_{2,2q}$ such that $g(\gamma_{2,1,2q})=f(\delta_1)$ (cf.~Figure \ref{labelbranches}). 
As $\Gamma(\alpha_{i_0})$ is a regular satellite, $g(\gamma_{i_0})\not=f(\delta_{j_0})$, and hence $\gamma_{i_0,j_0,0}\not=\gamma_{i_0}$. It follows that the monodromy relation associated with the line $L_{\gamma_{i_0,j_0,0}}$ is a simple tangent relation (cf.~(\ref{nearsimple})). Precisely, this relation is given by
\begin{align}\label{rel20}
\xi_{j_0,0}=\xi_{j_0+1,k_{j_0}},
\end{align}
where $k_{j_0}$ is some integer depending only on the first ordering of the elements $\xi_{j_0+1,r_{j_0+1}}$, $0\leq r_{j_0+1} \leq \nu_{j_0+1}-1$ (cf.~Figures~\ref{graphs} and \ref{tangentrelation-case2}). The graphs in Figure~\ref{graphs} are the real graphs of $f$ and $g$ in neighbourhoods of the intervals $[\alpha_{i_0},\alpha_{i_0+1}]$ and $[\beta_{j_0},\beta_{j_0+1}]$, respectively; $c(\xi_{j_0,0})$ and $c(\xi_{j_0+1,k_{j_0}})$ are the centres of the lassos $\xi_{j_0,0}$ and $\xi_{j_0+1,k_{j_0}}$, respectively. The picture on the left-side (respectively, right-side) of Figure \ref{tangentrelation-case2} represents the generators in a neighbourhood of $\beta_{j_0}$ and $\beta_{j_0+1}$ (respectively, in a neighbourhood of $\delta_{j_0}$) at $x=\alpha_{i_0}^+$ (respectively, at $x=\gamma_{i_0,j_0,0}^-$). 

Actually, as $\Gamma(\alpha_{i_0})$ is regular, $\gamma_{i_0,j_0,2q}\notin\{\gamma_{i_0-1},\gamma_{i_0}\}$ for any $0\leq q\leq \lambda_{i_0}-1$, and the monodromy relation associated with the special line~$L_{\gamma_{i_0,j_0,2q}}$ is a simple tangent relation given by
\begin{align}\label{rel2q}
\xi_{j_0,-q}=\xi_{j_0+1,k_{j_0}-q}.
\end{align}
This follows immediately from (\ref{rel20}) and the following observation.

\begin{observation}\label{fundamentalprinciple}
For any $i$ ($1\leq i\leq m$), when $x$ moves on the circle $\vert x-\alpha_{i}\vert=\varepsilon$ by the angle $2\pi/\lambda_{i}$, the centre of each lasso $\xi_{j,r_j}$ ($1\leq j\leq\ell$, $0\leq r_j\leq\nu_{j}-1$) turns on the circle $\vert y-\beta_j\vert = \varepsilon^{\lambda_{i}/\nu_{j}}$ by the angle $2\pi/\nu_{j}$.
\end{observation}

Similarly, if $f(\delta_{j_0})<0$, then, for each $0\leq q\leq \lambda_{i_0}-1$, there exists an~unique vertex $\gamma_{i_0,j_0,2q+1}\in B_{i_0,2q+1}$ such that $g(\gamma_{i_0,j_0,2q+1})=f(\delta_{j_0})$, and by the~same argument as above, the monodromy relation associated with the~line $L_{\gamma_{i_0,j_0,2q+1}}$ is given by
\begin{align}\label{rel2q+1}
\xi_{j_0,h_{j_0}-q}=\xi_{j_0+1,k_{j_0}-q},
\end{align}
where $h_{j_0}$ and $k_{j_0}$ are integers depending only on the first ordering of the elements $\xi_{j_0,r_{j_0}}$ ($0\leq r_{j_0}\leq \nu_{j_0}-1$) and $\xi_{j_0+1,r_{j_0+1}}$ ($0\leq r_{j_0+1} \leq \nu_{j_0+1}-1$). 

\begin{remark}
Note that the relations (\ref{rel2q}) can also be written under the form (\ref{rel2q+1}) by taking $h_{j_0}=0$. 
\end{remark}

Combined with (\ref{rel.alpha.i0}), the relations (\ref{rel2q+1}) imply 
\begin{align*}
\xi_{j_0,h_{j_0}-k}=\xi_{j_0+1,k_{j_0}-k} \quad \mbox{for} \quad k\in\mathbb{Z},
\end{align*}
and therefore,
\begin{align*}
\xi_{j_0,k}=\xi_{j_0+1,k+(k_{j_0}-h_{j_0})} \quad \mbox{for} \quad k\in\mathbb{Z}.
\end{align*}
Then, by reordering the generators $\xi_{j,k}$ successively for $j=2,\ldots,\ell$, we can assume that
\begin{align}\label{rel.gamma.i0k.bis}
\xi_{j_0,k}=\xi_{j_0+1,k} \quad \mbox{for} \quad k\in\mathbb{Z},
\end{align}
and hence, as $j_0$ is arbitrary, we can take, as generators, the elements
\begin{align}
\xi_{k}:=\xi_{j_0,k} \quad \mbox{for} \quad k\in\mathbb{Z}.
\end{align}
Then, the relations (\ref{rel.alpha.i0}) are written as
\begin{align}\label{rel.alpha.i0.bis}
\xi_{k}=\xi_{k+\lambda_{i_0}} 
\quad \mbox{for} \quad k\in\mathbb{Z},
\end{align}
and, by applying Proposition \ref{prop26}, we have
\begin{align}\label{relnu0}
\xi_{k+\nu_0}=\omega\xi_k\omega^{-1} \quad \mbox{for} \quad k\in\mathbb{Z},
\end{align}
where $\omega:=\xi_{\nu_0-1}\ldots\xi_0$. Indeed, by Bezout's identity, there exist $k_1,\ldots,k_{\ell}\in\mathbb{Z}$ such that $\nu_0=k_1\nu_1+\ldots+k_\ell\nu_\ell$. Then, by (\ref{rel.1}), $$\xi_{k+\nu_0}=(\omega_\ell^{k_\ell}\ldots\omega_1^{k_1})\cdot\xi_k\cdot (\omega_\ell^{k_\ell}\ldots\omega_1^{k_1})^{-1},$$
while Proposition \ref{prop26} shows that 
$$\omega_\ell^{k_\ell}\ldots\omega_1^{k_1}=\xi_{\nu_0-1}\ldots\xi_0.$$

\begin{remark}
The relations (\ref{rel.alpha.i0.bis}) and (\ref{relnu0}) associated with the regular satellite $\Gamma(\alpha_{i_0})$ imply that the fundamental group $\pi_1({\mathbb{C}^2\setminus C})$ is a quotient of the group $G(\nu_0;\lambda_{i_0})$.
\end{remark}

\begin{figure}[H]
\includegraphics[width=7.4cm,height=1.5cm]{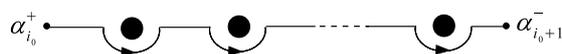}
\caption{The modified line segment $(\alpha_{i_0}^+,\alpha_{i_0+1}^-)$}
\label{modifiedsegment}
\end{figure}

\begin{figure}[H]
\includegraphics[width=11cm,height=5.3cm]{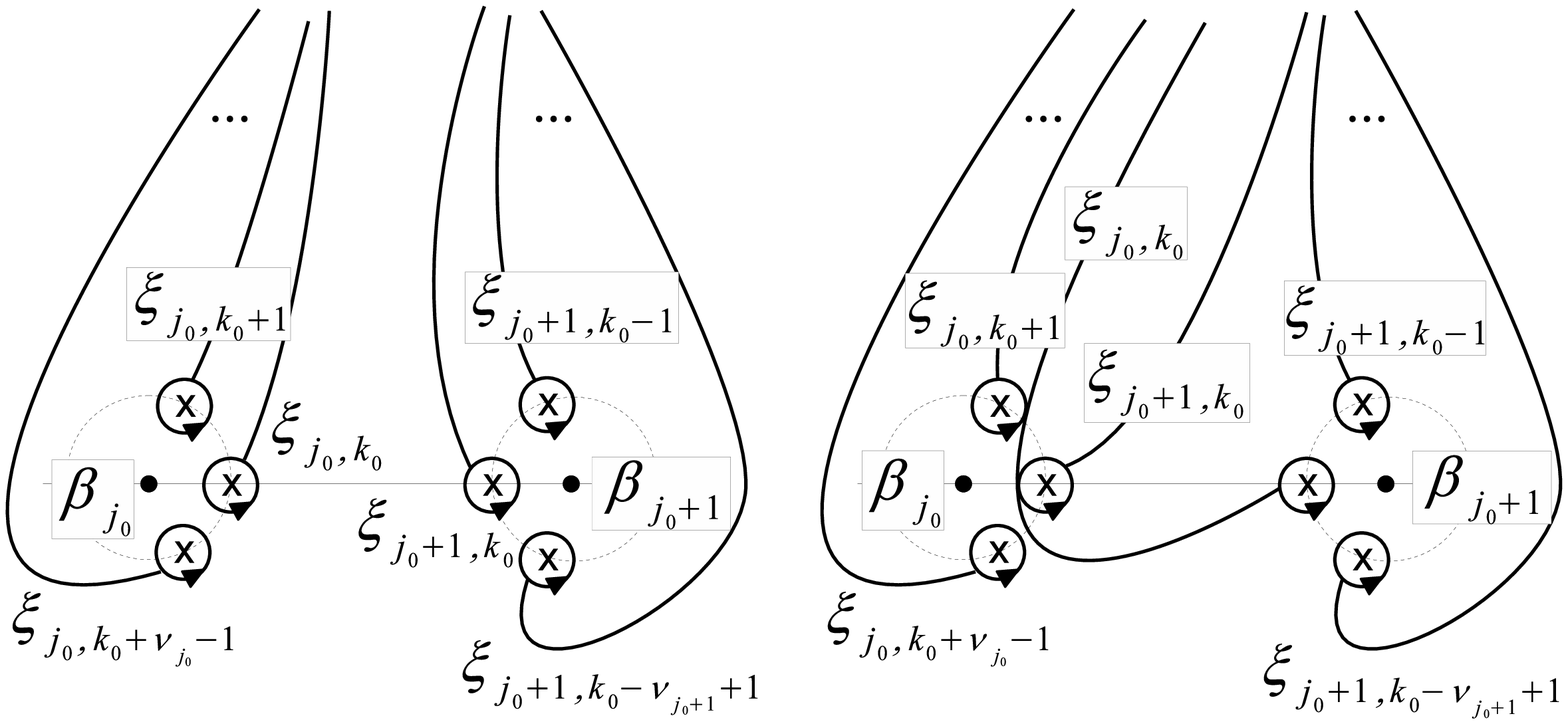}
\caption{Deformation of the generators when $0<f(\delta_{j_0})<g(\gamma_{i_0})$}
\label{deformationgenerators}
\end{figure}

Now, let us consider the monodromy relations associated with the other satellites. For simplicity, we still assume $g(\gamma_{i_0})>0$. 
We start with the satellite $\Gamma(\alpha_{i_0+1})$ and first look for the relations around the line $L_{\alpha_{i_0+1}}$. For this purpose, we need to know how the generators are deformed when $x$ moves along the `modified' line segment $(\alpha_{i_0}^+,\alpha_{i_0+1}^-)$. Here, `modified' means that $x$ makes a half-turn counterclockwise around each vertex of $\Gamma\cap (\alpha_{i_0}^+,\alpha_{i_0+1}^-)$ corresponding to a special line (cf.~Figure \ref{modifiedsegment}).
Take an element $j_0$ such that $1\leq j_0\leq \ell-1$. If $0<f(\delta_{j_0})<g(\gamma_{i_0})$, then there are exactly two vertices $\gamma_{i_0,j_0,0}\not=\gamma_{i_0}$ and $\gamma_{i_0+1,j_0,2q_0}\not=\gamma_{i_0}$ (for some $0\leq q_0\leq\lambda_{i_0+1}-1$) on the line segment $(\alpha_{i_0}^+,\alpha_{i_0+1}^-)$ that correspond to the special lines of the pencil associated with the critical value $f(\delta_{j_0})$ (i.e., $g(\gamma_{i_0,j_0,0})=f(\delta_{j_0})$ and $g(\gamma_{i_0+1,j_0,2q_0})=f(\delta_{j_0})$). The first one $\gamma_{i_0,j_0,0}$ is in $\Gamma(\alpha_{i_0})$ and the second one $\gamma_{i_0+1,j_0,2q_0}$ is in $\Gamma(\alpha_{i_0+1})$. 
Therefore, when $x$ moves along the modified line segment $(\alpha_{i_0}^+,\alpha_{i_0+1}^-)$, the generators are deformed as in Figure~\ref{deformationgenerators}. The picture on the left-side of the figure represents the generators at $x=\alpha_{i_0}^+$ (i.e., before the deformation). The picture on the right-side represents the generators at $x=\alpha_{i_0+1}^-$ (i.e., after the deformation). 
However, by (\ref{rel.gamma.i0k.bis}), we can suppose that the generators in the fibre $x=\alpha_{i_0+1}^-$ are still the same as in the fibre $x=\alpha_{i_0}^+$. In other words, the picture on the left-side of Figure~\ref{deformationgenerators} also represents the generators at $x=\alpha_{i_0+1}^-$.
Hence, by the same argument as above, the monodromy relations associated with the special line $L_{\alpha_{i_0+1}}$ give the relations
\begin{align}\label{relalphai0+1}
\xi_{k}=\xi_{k+\lambda_{i_0+1}} \quad\mbox{for}\quad k\in\mathbb{Z}.
\end{align}

We get the same relations if $g(\gamma_{i_0})<f(\delta_{j_0})$ or if $f(\delta_{j_0})<0$. Indeed, in these two cases, the set $g^{-1}(f(\delta_{j_0}))\cap (\alpha_{i_0}^+,\alpha_{i_0+1}^-)$ is empty, and therefore the configuration of the generators is identical on the fibres $x=\alpha_{i_0+1}^-$ and $x=\alpha_{i_0}^+$.

\begin{figure}[H]
\includegraphics[width=11cm,height=5.3cm]{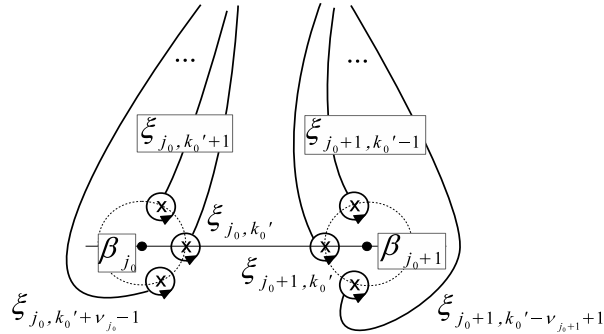}
\caption{Generators at $x=\alpha_{i_0+1}^+$ and $x=\alpha_{i_0+2}^-$ when $g(\gamma_{i_0})$, $g(\gamma_{i_0+1})$ and $f(\delta_{j_0})$ are $>0$}
\label{genhalfturn}
\end{figure}

The monodromy relations around the special lines corresponding to the vertices located on the branches of $\Gamma(\alpha_{i_0+1})$ do not give any new relation. This can be directly shown easily but it is not necessary. In fact, as we shall see below, it suffices to collect the monodromy relations associated with the special lines $L_{\alpha_{i}}$ for all $i$, $1\leq i\leq m$. We already know that, for $i=i_0$ and $i_0+1$, the monodromy relations around $L_{\alpha_{i}}$ are given by $\xi_{k}=\xi_{k+\lambda_{i}}$ for all $k\in\mathbb{Z}$. In fact, this is true for any $i$. For instance, let us show it for $i=i_0+2$.
For this purpose, we need to know how the generators are deformed when $x$ makes a half-turn on the circle $\vert x-\alpha_{i_0+1}\vert=\varepsilon$ from $\alpha_{i_0+1}^-$ to $\alpha_{i_0+1}^+$, and then moves along the modified line segment $(\alpha_{i_0+1}^+,\alpha_{i_0+2}^-)$.
Again, take $j_0$ such that $1\leq j_0\leq\ell-1$, and for simplicity assume that $g(\gamma_{i_0})$, $g(\gamma_{i_0+1})$ and $f(\delta_{j_0})$ are positive. (The other cases are similar and left to the reader.)
By Observation~\ref{fundamentalprinciple}, when $x$ makes a half-turn on the circle $\vert x-\alpha_{i_0+1}\vert=\varepsilon$ from $\alpha_{i_0+1}^-$ to $\alpha_{i_0+1}^+$, the generators are deformed as in Figure~\ref{genhalfturn}, where $k'_0\in\mathbb{Z}$. That is, the configuration of the generators on the fibre $x=\alpha_{i_0+1}^+$ is just the parallel translation of that on the fibre $x=\alpha_{i_0+1}^-$. When $x$ moves along the modified line segment $(\alpha_{i_0+1}^+,\alpha_{i_0+2}^-)$, the generators are still as in Figure \ref{genhalfturn}. The proof of this fact is as above except that now $\gamma_{i_0+1}\in (\alpha_{i_0+1}^+,\alpha_{i_0+2}^-)$ may correspond to an exceptional singularity (in particular, we may have $f(\delta_{j_0})=g(\gamma_{i_0+1})$). But in this case, when $x$ makes a half-turn on the circle $\vert x-\gamma_{i_0+1}\vert=\varepsilon$ from $\gamma_{i_0+1}^-$ to $\gamma_{i_0+1}^+$, the generators remains unchanged, as, by~(\ref{rel.gamma.i0k.bis}), $\xi_{j_0,k'_0}=\xi_{j_0+1,k'_0}$.
Finally, exactly as above, the monodromy relations associated with the special line $L_{\alpha_{i_0+2}}$ give the relations
\begin{align*}
\xi_{k}=\xi_{k+\lambda_{i_0+2}} 
\quad\mbox{for}\quad k\in\mathbb{Z}.
\end{align*}

This argument can be repeated for all the other values of $i$, $1\leq i\leq m$, so that the monodromy relations associated with the special line $L_{\alpha_{i}}$ for any $i$, $1\leq i\leq m$, are written as
\begin{align}\label{rel.alpha.i}
\xi_{k}=\xi_{k+\lambda_{i}} 
\quad\mbox{for}\quad k\in\mathbb{Z}.
\end{align}

By Proposition \ref{prop29}, the collection of relations (\ref{rel.alpha.i}), for $1\leq i\leq m$, and the relation (\ref{relnu0}) are equivalent to 
\begin{equation*}\label{rel.alli}
\left\{
\begin{aligned}
& \xi_{k}=\xi_{k+\lambda_{0}} \\
& \xi_{k+\nu_0}=\omega\xi_k\omega^{-1} 
\end{aligned}
\right.
\quad \mbox{for} \quad k\in\mathbb{Z}.
\end{equation*}
In particular, this means that the fundamental group $\pi_1(\mathbb{C}^2\setminus C)$ is presented by the generators $\xi_k$ ($k\in\mathbb{Z}$) and $\omega$ and by a set of relations that includes the following relations:
\begin{align}
\label{eq1} & \omega=\xi_{\nu_0-1}\ldots\xi_0,\\
\label{eq2} & \xi_{k+\lambda_{0}}=\xi_{k}\quad (k\in\mathbb{Z}),\\
\label{eq3} & \xi_{k+\nu_{0}}=\omega\xi_{k}\omega^{-1}\quad (k\in\mathbb{Z}).
\end{align}
In other words, $\pi_1(\mathbb{C}^2\setminus C)$ is a quotient of the group $G(\nu_0;\lambda_0)$. 

Now, consider the family $\{C_t\}_{0\leq t\ll 1}$ of $\mathbb{R}$-join type curves, where $C_t$ is defined by the equation
\begin{displaymath}
f(y)=(1-t)g(x).
\end{displaymath}
For any $0<t\ll 1$, it is easy to see that the curve $C_t$ has only inner singularities. 
Therefore, by the degeneration principle \cite{O3,Z}, for any sufficiently small $t>0$, there is a canonical epimorphism 
\begin{displaymath}
\psi_t\colon \pi_1(\mathbb{C}^2\setminus C)=\pi_1(\mathbb{C}^2\setminus C_0) \twoheadrightarrow \pi_1(\mathbb{C}^2\setminus C_t)\simeq G(\nu_0;\lambda_0).
\end{displaymath}
Let us recall briefly how $\psi_t$ is defined. Let $L_\infty$ be the line at infinity, and set $C':=C\cup L_\infty$ and $C'_t:=C_t\cup L_{\infty}$. Pick a sufficiently small regular neighbourhood $N$ of $C'$ in $\mathbb{P}^2$ so that $\imath\colon\mathbb{P}^2\setminus N\hookrightarrow \mathbb{P}^2\setminus C'=\mathbb{C}^2\setminus C$ is a homotopy equivalence, and take a sufficiently small $t$ so that $C_t'$ is contained in $N$. Then, $\psi_t$ is defined by taking the composition of ${\imath_\sharp}^{-1} \colon \pi_1(\mathbb{C}^2\setminus C) \to \pi_1(\mathbb{P}^2\setminus N)$ with the homomorphism induced by the inclusion $\mathbb{P}^2\setminus N\hookrightarrow \mathbb{P}^2\setminus C'_t = \mathbb{C}^2\setminus C_t$.
To distinguish the generators, we write $\xi_k(t)$ ($k\in \mathbb{Z}$) for the generators of $\pi_1(\mathbb{C}^2\setminus C_t)$, which are represented by the same loops as $\xi_k$. Note that $\psi_t(\xi_k)=\xi_k(t)$. As $C_t$ is generic, $\pi_1(\mathbb{C}^2\setminus C_t)$ is presented by the generators $\xi_k(t)$ ($k\in\mathbb{Z}$) and $\omega(t):=\xi_{\nu_0-1}(t)\ldots\xi_0(t)$ and by the relations (\ref{eq1})--(\ref{eq3}), replacing $\xi_k$ by $\xi_k(t)$ and $\omega$ by $\omega(t)$. (We may assume that $N\cap L_{\alpha_{i_0}^+}$ is a copy of $d$ disjoint (topologically) tiny $2$-disks so that $\xi_k$ ($0\leq k\leq d-1$) are free generators of $\pi_1(L_{\alpha_{i_0}^+}\setminus N)$.)
This implies that $\ker \psi_t$ is trivial, and hence
\begin{displaymath}
\pi_1(\mathbb{C}^2\setminus C)\simeq G(\nu_0;\lambda_0).
\end{displaymath}
(In particular, the branches of~the satellites $\Gamma(\alpha_i)$, $i\not=i_0$, do not give any new relation.)

As for the fundamental group $\pi_1(\mathbb{P}^2\setminus \widetilde C)$, we proceed as follows. If $d\geq d'$, then the base locus of the pencil $X=\gamma Z$ ($\gamma\in\mathbb{C}$) in $\mathbb{P}^2$ does not belong to the curve, and therefore the group $\pi_1(\mathbb{P}^2\setminus \widetilde C)$ is obtained from the above presentation of $\pi_1(\mathbb{C}^2\setminus C)$ by adding the vanishing relation at infinity $\omega_1\ldots\omega_\ell=e$. By Proposition \ref{prop26}, the relations (\ref{eq1}) and (\ref{eq3}) imply $\omega_j=\omega^{\nu_j/\nu_0}$ ($1\leq j\leq \ell$). Therefore, the relation $\omega_1\ldots\omega_\ell=e$ can also be written~as 
\begin{displaymath}
\omega^{d/\nu_0}=e.
\end{displaymath}
It follows that $\pi_1(\mathbb{P}^2\setminus \widetilde C)\simeq G(\nu_0;\lambda_0;d/\nu_0)$.

If $d'\geq d$, then we consider again the family $\{C_t\}_{0\leq t\ll 1}$, where $C_t$ is defined by the equation $f(y)=(1-t)g(x)$. We use the same regular neighbourhood $N$ and the same isomorphism $\psi_t\colon \pi_1(\mathbb{C}^2\setminus C)\to \pi_1(\mathbb{C}^2\setminus C_t)$ for a sufficiently small $t>0$. To compute $\pi_1(\mathbb{C}^2\setminus C)$, this time we consider the pencil given by the horizontal lines $y=\delta$, where $\delta\in\mathbb{C}$. We fix a generic line $y=\delta_0$ and we choose geometric generators $\rho_k$ ($0\leq k\leq d'-1$) as above so that the $\rho_k$'s give generators of the fundamental group of the generic fibre of each complement $\mathbb{P}^2\setminus N$, $\mathbb{C}^2\setminus C$ and $\mathbb{C}^2\setminus C_t$ simultaneously. Then we define elements $\tau$ and $\rho_k$, for $k\in\mathbb{Z}$, in the same way as we defined the elements $\omega$ and $\xi_k$ ($k\in\mathbb{Z}$) above. Now, as $\psi_t$ is an isomorphism and $C_t$ is generic, the generating relations for each group $\pi_1(\mathbb{C}^2\setminus C_t)$, $\pi_1(\mathbb{C}^2\setminus C)$ and $\pi_1(\mathbb{P}^2\setminus N)$ are given by
\begin{align*}
& \tau=\rho_{\lambda_0-1}\ldots\rho_0,\\
& \rho_{k+\nu_{0}}=\rho_{k}\quad (k\in\mathbb{Z}),\\
& \rho_{k+\lambda_{0}}=\tau\rho_{k}\tau^{-1}\quad (k\in\mathbb{Z}).
\end{align*}
As $d'\geq d$, the base locus of the pencil $Y=\delta Z$ ($\delta\in\mathbb{C}$) in $\mathbb{P}^2$ does not belong to $\widetilde C$, and hence the group $\pi_1(\mathbb{P}^2\setminus \widetilde C)$ is obtained from the above presentation of $\pi_1(\mathbb{C}^2\setminus C)$ by adding the vanishing relation at infinity $\tau^{d'/\lambda_0}=e$. Finally, we get $\pi_1(\mathbb{P}^2\setminus \widetilde C)\simeq G(\lambda_0;\nu_0;d'/\lambda_0)$.

\section{Applications}

In this section, we give two applications of Theorem~\ref{mt}.

\subsection{Maximal irreducible nodal curves}

An irreducible curve is said to be \emph{nodal} if it has only
node singularities. By Pl\"ucker's formula, an irreducible nodal curve of degree $d$ has at most $(d-1)(d-2)/2$ nodes. An irreducible nodal curve is called \emph{maximal} if it has exactly $(d-1)(d-2)/2$ nodes (equivalently, if its genus is $0$). A method to construct such curves is given in \cite{O1}. There, the irreducibility is obtained using the braid group action. Hereafter, we repeat the construction of \cite{O1} but apply Theorem \ref{mt} to show the irreducibility.

For simplicity, let us suppose that $d=2n+1$, $n\in\mathbb{Z}$. (The construction is similar when $d$ is even.) Consider the \emph{Chebyshev} polynomial of degree $d$, which is defined by $T_d(z):=\cos(d\arccos(z))$. This polynomial has $2n+1$ simple real roots $\beta_1<\ldots<\beta_{2n+1}$ and $2n$ critical points $\delta_1,\ldots,\delta_{2n}$, with $\delta_j\in (\beta_j,\beta_{j+1})$, such that $T_d(\delta_1)=T_d(\delta_3)=\ldots=T_d(\delta_{2n-1})=1$ and $T_d(\delta_2)=T_d(\delta_4)=\ldots=T_d(\delta_{2n})=-1$.
Now, take a small deformation $\widetilde T_d(z)$ of $T_d(z)$ so that:
\begin{enumerate}
\item
$\widetilde T_d(z)$ has $n$ critical points $\gamma_1,\gamma_3,\ldots,\gamma_{2n-1}$ such that $\widetilde T_d(\gamma_1)=\widetilde T_d(\gamma_3)=\ldots=\widetilde T_d(\gamma_{2n-1})=1$;
\item
$\widetilde T_d(z)$ has $n-1$ critical points $\gamma_2,\gamma_4,\ldots,\gamma_{2n-2}$ such that $\widetilde T_d(\gamma_2)=\widetilde T_d(\gamma_4)=\ldots=\widetilde T_d(\gamma_{2n-2})=-1$;
\item
$\widetilde T_d(z)$ has a critical point $\gamma_{2n}$ with $\widetilde T_d(\gamma_{2n})<-1$.
\end{enumerate}
The existence of such a polynomial $\widetilde T_d(z)$ is due to R.~Thom \cite{T}. It has $2n+1$ simple real roots $\alpha_1<\ldots<\alpha_{2n+1}$, and $\gamma_i\in (\alpha_i,\alpha_{i+1})$. Then, consider the $\mathbb{R}$-join-type curve $C$ defined by $T_d(y)=\widetilde T_d(x)$. Clearly, the satellite $\Gamma(\alpha_{2n+1})$ (of the bifurcation graph of $C$ with respect to $\widetilde T_d$) is regular and the numbers $\nu_0,\lambda_0$ which appear in Theorem \ref{mt} are equal to $1$. Hence, $\pi_1(\mathbb{C}^2\setminus C)\simeq\mathbb{Z}$ and $\pi_1(\mathbb{P}^2\setminus \widetilde C)\simeq\mathbb{Z}_d$. In particular, the curve $C$ is irreducible. The number of nodes is given by the cardinality of the set
\begin{align*}
\{(\gamma_{2i-1},\delta_{2j-1})\mid 1\leq i,j\leq n\} \cup
\{(\gamma_{2i},\delta_{2j})\mid 1\leq i\leq n-1,\, 1\leq j\leq n\},
\end{align*}
that is,
\begin{align*}
n^2+(n-1)n=\frac{(d-1)(d-2)}{2}.
\end{align*}
In other words, $C$ is a maximal irreducible nodal curve.

\subsection{Curves with node and cusp singularities}

Let $C\colon f(y)=g(x)$ be an $\mathbb{R}$-join-type curve with only nodes and cusps as singularities. For simplicity, we suppose that $C$ has degree $d=6n$, $n\in\mathbb{Z}$. The maximal number of cusps on such a curve is obtained when $f$ and $g$ have the form $f(y)=a(y-\beta_1)^3(y-\beta_2)^3\ldots(y-\beta_{2n})^3$ and $g(x)=b(x-\alpha_1)^2(x-\alpha_2)^2\ldots(x-\alpha_{3n})^2$, in which case we have $6n^2$ cusps. (As usual, we suppose $\beta_j<\beta_{j+1}$ and $\alpha_i<\alpha_{i+1}$ for all $i,j$.) By the result of R.~Thom \cite{T}, $f$ can be chosen so that its $2n-1$ critical points $\delta_1<\ldots<\delta_{2n-1}$ satisfy $f(\delta_1)=f(\delta_3)=\ldots=f(\delta_{2n-1})=-1$ and $f(\delta_2)=f(\delta_4)=\ldots=f(\delta_{2n-2})=1$. Similarly, $g$ can be chosen so that
its $3n-1$ critical points $\gamma_1<\ldots<\gamma_{3n-1}$ satisfy $g(\gamma_1)=g(\gamma_2)=\ldots=g(\gamma_{3n-2})=-1>g(\gamma_{3n-1})$. In this case, the curve also has $n(3n-2)$ nodes. Clearly, the satellite $\Gamma(\alpha_{3n})$ (of the bifurcation graph of $C$ with respect to $g$) is regular, and hence, by Theorem \ref{mt}, $\pi_1(\mathbb{C}^2\setminus C)\simeq G(3;2)\simeq B(3)$ (the braid group on $3$ strings), while $\pi_1(\mathbb{P}^2\setminus \widetilde C)\simeq G(3;2;2n)$. Note that, by Theorem (2.12) of \cite{O}, we have a central extension
$$0\rightarrow\mathbb{Z}_n\rightarrow G(3;2;2n) \rightarrow \mathbb{Z}_3*\mathbb{Z}_2 \rightarrow 0,$$
where $\mathbb{Z}_n$ is generated by $\omega^2$.

\end{document}